\documentclass[12pt]{article}
\usepackage{amssymb,amsmath,amsthm}

\begin{document}

\title{Optimal additive quaternary codes of low dimension}
\author{J\"urgen Bierbrauer\\
Department of Mathematical Sciences\\
Michigan Technological University\\
Houghton, Michigan 49931 (USA) \\ \\
S. Marcugini\footnote{The research of  S. Marcugini and F.~Pambianco was
 supported in part by the Italian
National Group for Algebraic and Geometric Structures and their Applications (GNSAGA - INDAM) and by
University of Perugia (Project: Curve, codici e configurazioni di punti, Base Research
Fund 2018).}    \,and F. Pambianco$^*$\\
Dipartimento di Matematica e Informatica\\
Universit\`a degli Studi di Perugia\\
Perugia (Italy) }

\maketitle
\newtheorem{Theorem}{Theorem}
\newtheorem{Proposition}{Proposition}
\newtheorem{Lemma}{Lemma}
\newtheorem{Definition}{Definition}
\newtheorem{Corollary}{Corollary}
\newtheorem{Example}{Example}
\def\nz{\mathbb{N}}
\def\gz{\mathbb{Z}}
\def\rz{\mathbb{R}}
\def\ef{\mathbb{F}}
\def\CC{\mathbb{C}}
\def\o{\omega}
\def\p{\overline{\omega}}
\def\e{\epsilon}
\def\a{\alpha}
\def\b{\beta}
\def\g{\gamma}
\def\d{\delta}
\def\l{\lambda}
\def\s{\sigma}
\def\bsl{\backslash}
\def\la{\longrightarrow}
\def\arr{\rightarrow}
\def\ov{\overline}
\def\sm{\setminus}
\newcommand{\D}{\displaystyle}
\newcommand{\T}{\textstyle}

\begin{abstract}
An additive quaternary $[n,k,d]$-code (length $n,$ quaternary dimension $k,$ minimum distance $d$)
is a $2k$-dimensional $\ef_2$-vector space of $n$-tuples with entries in $Z_2\times Z_2$
(the $2$-dimensional vector space over $\ef_2$) with minimum Hamming distance $d.$
We determine the optimal parameters of additive quaternary codes of dimension $k\leq 3.$
The most challenging case is dimension $k=2.5.$ We prove that 
an additive quaternary $[n,2.5,d]$-code where $d<n-1$ exists if and only if 
$3(n-d)\geq \lceil d/2\rceil +\lceil d/4\rceil +\lceil d/8\rceil$.
In particular we construct new optimal $2.5$-dimensional additive quaternary codes.
As a by-product we give a direct proof for the fact that a binary linear $[3m,5,2e]_2$-code for $e<m-1$ exists if and only if
the Griesmer bound $3(m-e)\geq \lceil e/2\rceil +\lceil e/4\rceil+\lceil e/8\rceil$ is satisfied.
\end{abstract}


\section{Introduction}
\label{introsection}

The concept of additive codes is a far-reaching and natural generalization of linear codes,
see \cite{book2nded}, Chapter 18. Here we restrict to the quaternary case.

\begin{Definition}
\label{addquatdef}
Let $k$ be such that $2k$ is a positive integer. An additive quaternary $[n,k]$-code $C$ 
(length $n,$ dimension $k$) is a $2k$-dimensional subspace of $\ef_2^{2n}$
where the coordinates come in pairs of two. We
view the codewords as $n$-tuples where the coordinate entries are elements of $\ef_2^{2}$
and use the Hamming distance.
\end{Definition}

 We write the parameters of the code as $[n,k,d]$ where $d$ is the minimum Hamming distance.
Here $k$ is the quaternary dimension. As an example, in case $k=2.5$ the code is a $5$-dimensional
vector space over $\ef_2.$
Additive codes are particularly interesting because of a link to quantum stabilizer codes, see 
\cite{BFGMPquantgeom,noquant1354,CRSS}.
We will also use the geometric construction of additive quaternary codes. In fact, a quaternary $[n,k,d]$-code
is equivalent to a multiset of $n$ lines in $PG(2k-1,2)$ such that each hyperplane of $PG(2k-1,2)$
contains at most $s=n-d$ of those lines, in the multiset sense.
Blokhuis and Brouwer \cite{BB} first suggested the problem of determining the optimum parameters
of additive quaternary codes. In earlier work we determined all such optimal parameters when $n\leq 13,$
see \cite{book2nded}, Chapter 18 and \cite{addJCTA08}. For further results concerning larger lengths see
\cite{no1559,addellipt}.
In the present work we determine all optimal parameters when the
quaternary dimension is $k\leq 3.$ Dimensions $k\leq 2$ are degenerate cases, see Section \ref{k=2section}.
Dimension $3$ is easily dealt with as well, see Section \ref{k=3section}. Our main result is Theorem \ref{k=2.5theorem}
in Section \ref{k=2.5section} where the optimal parameters of $2.5$-dimensional additive quaternary codes are
determined.
For $k>1$ we prefer to work with the {\bf species} $s=n-d$ instead of the minimum distance $d.$
Define $n_k(s)$ to be the maximal length $n$ such that an additive $[n,k,n-s]$-code exists.
For integer $k,$ let $n_{k,lin}(s)$ be the maximal $n$ such that a linear quaternary $[n,k,n-s]_4$-code exists.
In the present paper we determine $n_k(s)$ for $k\leq 3$ and all $s.$
The following obvious lemma will be used to prove nonexistence results:

\begin{Lemma}
\label{concatlemma}
The concatenation of a quaternary additive $[n,k,d]$-code and the binary linear $[3,2,2]_2$-code is a
binary linear $[3n,2k,2d]_2$-code. 
\end{Lemma}

\section{Dimensions $k\leq 2.$}
\label{k=2section}

Clearly dimension $k=1$ is a trivial case, the optimal parameters being $[n,1,n].$
Dimension $k=1.5$ is degenerate as well. The ambient space is the Fano plane and the optimal choice is to use each of
its seven lines with multiplicity $s.$ This shows $n_{1.5}(s)=7s.$ The corresponding codes have parameters $[7s,1.5,6s].$
Dimension $k=2$ still is degenerate. In the linear case we have $n_{2,lin}(s)=5s.$ In fact we work in the projective line
$PG(1,4)$ and the optimal choice is to use each of its points with multiplicity $s.$ 

\begin{Proposition}
\label{k=2prop}
We have $n_2(s)=n_{2,lin}(s)=5s$ for all $s.$
\end{Proposition}
\begin{proof}
Assume $n_2(s)>5s.$ We would have a $[5s+1,2,4s+1]$-code. Lemma \ref{concatlemma} would yield a binary linear
$[15s+3,4,8s+2]_2$-code. This contradicts the Griesmer bound.
\end{proof}  
    
\section{The case of dimension $k=3.$}
\label{k=3section}

The optimal parameters of linear quaternary $3$-dimensional codes are of course known: 

\begin{Proposition}
\label{k=3linprop}
We have $n_{3,lin}(2)=6, n_{3,lin}(3)=9, n_{3,lin}(4)=16,$
$$n_{3,lin}(5i)=21i, n_{3,lin}(5i+1)=21i+1\mbox{ and }
n_{3,lin}(5i+\s )=21i+1+5(\s -1)$$ 
for $i\geq 1, \s\in\lbrace 2,3,4\rbrace .$
\end{Proposition}
\begin{proof}
For $d<9$ this is easy to check. For larger $d$ we can invoke a result by Hamada-Tamari \cite{HT}
stating that linear $[n,3,d]_q$-codes for $d\geq (q-1)^2$ exist if and only if the parameters satisfy the Griesmer bound
(see \cite{book2nded}, Theorem 17.7). This coincides with the statement of our proposition.
 \end{proof}

\begin{Theorem}
\label{k=3theorem}
We have $n_3(s)=n_{3,lin}(s)$ for all $s.$
\end{Theorem}   
\begin{proof}
Assume there is an additive $3$-dimensional code with larger $n$ and the same species.
We illustrate with case $s=5i.$ We would have a $[21i+1,3,16i+1]$-code. Lemma \ref{concatlemma}
yields a linear $[63i+3,6,32i+2]_2$-code, which contradicts the Griesmer bound . The other cases are analogous.
\end{proof}
 
\section{The case of dimension $2.5.$}
\label{k=2.5section}

Our main result is the following:

\begin{Theorem}
\label{k=2.5theorem}
An additive quaternary $[n,2.5,d]$-code where $d<n-1$ exists if and only if 
$3(n-d)\geq \lceil d/2\rceil +\lceil d/4\rceil +\lceil d/8\rceil .$
\end{Theorem}

In the present section we prove Theorem \ref{k=2.5theorem}. In the sequel use the abbreviation $d_l=\lceil d/l\rceil .$
The necessity is obvious. In fact, Lemma \ref{concatlemma} applied to an additive quaternary $[n,2.5,d]$-code 
yields a binary linear $[3n,5,2d]_2$-code. The condition of Theorem \ref{k=2.5theorem} is the Griesmer bound as applied
to this binary code. It remains to prove sufficiency: given $d,n$ satisfying the condition of the theorem
we need to construct an additive quaternary $[n,2.5,d]$-code. As before, let $s=n-d.$
For each $s$ consider the pair $D_s=(s,m_s)$ where
$m_s$ is the maximal $n$ such that $n,d=n-s$ satisfy the condition in Theorem \ref{k=2.5theorem}.
We need to prove the existence of an  $[m_s,2.5,m_s-s]$-code, for all $s\geq 2.$
When such a code exists we say that we represented $D_s.$ Here are some examples:
$$D_2=(2,8), D_3=(3,11), D_4=(4,16), D_5=(5,21), D_6=(6,26), D_7=(7,31).$$
Let $C$ be an $[n,2.5,d]$-code and $C'$ the code obtained by increasing each line multiplicity
of $C$ by 1. As $PG(4,2)$ has 155 lines and $PG(3,2)$ has 35 lines we see that
$C'$ is an $[n+155,2.5,d+120]$-code. Concerning the bound of the theorem we observe that
$3(n-d)-d_2-d_4-d_8$ is invariant under the substitution $n\mapsto n+155, d\mapsto d+120.$
This shows that we need prove the existence of an $[n,2.5,d]$-code only for $n<155.$
This means that it suffices to  construct $D_2,D_3,\dots ,D_{35}=(35,155).$
Observe that there is an obvious sum construction which shows that the existence of codes
$[m_1,2.5,l_1]$ and $[m_2,2.5,l_2]$ implies the existence of an
$[m_1+m_2,2.5,l_1+l_2]$-code. This shows that if $D_{s_1}$ and $D_{s_2}$ can be
constructed then also $D_{s_1}+D_{s_2}$ can be constructed. We see now that it suffices to
construct $D_2,\dots ,D_7$ as the remaining $D_s,s\leq 35$ follow from the sum construction.
Here are some examples: 
$$D_8=D_6+D_2, D_9=D_7+D_2, D_{10}=D_5+D_5, D_{11}=D_9+D_2, D_{12}=D_6+D_6.$$
It remains to construct $D_2,\dots ,D_7.$ Now $D_2$ implies $D_4$ as $D_2+D_2=D_4$
and $D_5=(5,21)$ is constructed as there is even a linear $[21,3,16]_4$-code (corresponding to the points
of $PG(2,4)$). We are
reduced to construct $D_2,D_3,D_6,D_7.$
Now $D_2=(2,8)$ corresponds to a $[8,2.5,6]$-code. This is the Blokhuis-Brouwer construction \cite{BB,IEEEaddZ4}.
In the same context an $[11,2.5,8]$-code was constructed. This is a representation of $D_3=(3,11).$
We are finally reduced to construct $D_6$ and $D_7.$

\subsection*{A construction}
Consider a chain
$$l_0\subset E_0\subset H_0\subset PG(4,2)$$
where $l_0$ is a line, $E_0$ a plane and $H_0$ a solid (hyperplane) in $PG(4,2).$
Let ${\cal V}$ be a set of $8$ lines such that each point in
$E_0\sm l_0$ is on precisely two lines of ${\cal V},$ each point outside $H_0$ is on precisely one line of ${\cal V}.$
Also, let ${\cal E}$ be a set of $8$ lines partitioning the points outside $E_0$ (Blokhuis-Brouwer construction).

\begin{Definition}
\label{2pointfivemultidef}
Let $C(g,h,v,e)$ be the additive $2.5$-dimensional quaternary code described by the following multiset of lines:
line $l_0$ with multiplicity $g,$ the remaining lines of $E_0$ each with multiplicity $h,$
the lines of ${\cal V}$ with multiplicity $v$ and the lines of 
${\cal E}$ with multiplicity $e.$
\end{Definition}

Clearly $C(g,h,v,e)$ has length $n=g+6h+8v+8e.$ Let $m(P)$ be the number of codelines (including multiplicities)
that contain point $P.$ If $P\in l_0,$ then $m(P)=g+2h,$ if $P\in E_0\sm l_0$ then $m(P)=3h+2v.$
If $P\in H_0\sm E_0$ then $m(P)=e$ whereas points $P$ outside $H_0$ have $m(P)=v+e.$
For each hyperplane $H$ let $m(H)=\sum_{P\in H}m(P).$ By double counting we obtain
$$s(H)=(m(H)-n)/2$$
where $s(H)$ (the species of $H$) is the number of codelines contained in $H.$
It follows that the numbers $n-s(H)$ are the nonzero weights of our code.
The numbers $m(H)$ and $s(H)$ are easy to determine:

\begin{Lemma}
\label{2pointfivemultimHlemma}
If $l_0\not\subset H$ then $m(H)=g+8h+12v+12e.$\\
If $l_0\subset H$ but $E_0\not\subset H$ then $m(H)=3g+6h+8v+12e.$\\
If $E_0\subset H\not=H_0$ then $m(H)=3g+18h+16v+8e.$\\
Finally $m(H_0)=3g+18h+8v+8e.$
\end{Lemma}
\begin{proof}
This is a trivial calculation. In the first case above $H$ has one point of $l_0,$ two further points
in $E_0,$ four further points in $H_0$ and finally $8$ affine points for a grand total
$m(H)=g+8h+4v+4e+8(v+e).$
In the second case $H$ contains three points on $l_0,$ no further point on $E_0,$ four further points on $H_0$
and eight affine points: $m(H)=3g+6h+4e+8(v+e).$ The remaining two cases are analogous.
\end{proof}

Our basic formula yields:

\begin{Corollary}
\label{2pointfivemultimHcor}
The nonzero weights of the codewords of $C(g,h,v,e)$ are
$$g+5h+6(v+e),6h+8v+6e,4v+8e,8(v+e).$$
$C(g,h,v,e)$ is an $[g+6h+8(v+e),2.5,d]_4$-code where 
$$d=Min(w_1=g+5h+6(v+e),w_2=6h+8v+6e,w_3=4v+8e).$$
\end{Corollary}

 We see that $C(2,0,1,2)$ is a $[26,2.5,20]$-code and
$C(1,1,0,3)$ is a $[31,2.5,24]$-code.
This completes the proof of Theorem \ref{k=2.5theorem}. 
Lemma \ref{concatlemma} yields

\begin{Corollary}
A binary linear $[3m,5,2e]_2$-code for $e<m-1$ exists if and only if
the Griesmer bound $3(m-e)\geq e_2+e_4+e_8$ is satisfied.
\end{Corollary}

\end{document}